\documentclass[11pt]{article}

\usepackage{amsmath,amssymb,amsfonts,latexsym,enumerate,setspace,verbatim,tocloft,rotating}
\usepackage{amsthm}
\usepackage{parskip}
\usepackage{setspace}
\usepackage[toc,titletoc]{appendix}
\usepackage{tabularx}
\usepackage{adforn}
\usepackage{thmtools}
\usepackage{tikz}
\usepackage{caption}
\usepackage[linewidth=1pt]{mdframed}
\usepackage{graphicx}
\usepackage[all]{nowidow}
\usepackage[margin=1.25in]{geometry}
\usepackage{placeins}

\newcommand{\R}{\mathbb{R}}
\newcommand{\Z}{\mathbb{Z}}
\newcommand{\C}{\mathbb{C}}
\newcommand{\ang}[1]{\langle #1 \rangle}
\newcommand{\ov}{\overline}

\theoremstyle{plain} 
 \newtheorem{theorem}{Theorem}[section]
\newtheorem*{theorem_th_nonoverlapping}{Theorem 1.1}
\newtheorem*{theorem_projections_of_Riesz_p}{Theorem 1.2}
\newtheorem*{theorem_th_symmetric1}{Theorem 1.3, part 1}
\newtheorem*{theorem_th_symmetric2}{Theorem 1.3, part 2}
 \newtheorem{proposition}[theorem]{Proposition}
 \newtheorem{lemma}[theorem]{Lemma}
 
 \newtheorem{remark}[theorem]{Remark}

\declaretheorem[style=definition,qed=$\blacktriangle$]{definition}

\newenvironment{lemmaproof}{
  
  \begin{proof}[Proof of Lemma]}{\end{proof}}

\title{Spectral Properties of\\Some Unions of Linear Spaces}
\author{Chun-Kit Lai, San Francisco State University
\and Bochen Liu, National Center for Theoretical Sciences, Taiwan 
\and Hal Prince, San Francisco State University}

\begin{document}
\maketitle

\begin{abstract}
We consider \textit{additive spaces}, consisting
of two intervals of unit length or two general probability measures on $\R^1$, positioned on the
axes in $\R^2$, with a natural additive measure $\rho$.
We study the relationship between the exponential frames,
Riesz bases, and orthonormal bases of $L^2(\rho)$ and
those of its component spaces. We find that the existence of exponential bases depends strongly on how we position our measures on $\R^1$.
We show that non-overlapping additive spaces possess Riesz bases, and we give a necessary condition for overlapping spaces. We also show that some overlapping additive spaces of Lebesgue type
have exponential orthonormal bases, while some do not. 
A particular example is the \textbf{L} shape at the origin, which
has a unique orthonormal basis up to translations of the form
\[
\left\{e^{2 \pi i (\lambda_1 x_1 + \lambda_2 x_2)} : (\lambda_1, \lambda_2) \in \Lambda 
\right\},
\]
where
\[
\Lambda = \{ (n/2, -n/2) \mid n \in \Z \}.
\]
\end{abstract}

\section{Introduction}

A basic fact of Fourier analysis is that the unit interval $I$ in $\R$ has an exponential
orthonormal basis $\{ e^{2 \pi i n x} \}_{n \in \Z}$ for its $L^2(I)$. This concept was generalized to spectral sets by Fuglede \cite{Fug74}. A {\it spectral set} $\Omega$ (with Lebesgue measure normalized to 1) is a set in which there exists an exponential
orthonormal basis $E(\Lambda): = \{ e^{2 \pi i \lambda \cdot x} \}_{\lambda \in \Lambda}$ for its $L^2(\Omega)$ for some countable set $\Lambda$, which we call it a {\it spectrum}. The concept was further generalized to {\it spectral measures}, where a probability measure $\mu$ is spectral if there exists an exponential orthonormal basis for $L^2(\mu)$. 

\medskip

Spectral sets and measures have a long history.  Fuglede proposed a conjecture that spectral sets and translational tiles are equivalent. The conjecture was  open for thirty years until it was disproved by Tao \cite{Tao04} in dimension five or higher. It was later disproved in both directions in dimension three or higher \cite{KM06,KM}. Nonetheless, the study of the conjecture remains active until today. For example, Lev and Matolcsi recently proved that Fuglede's conjecture is true if the set is convex \cite{LM19}.  Concerning spectral measures, Jorgensen and Pedersen showed the middle-fourth Cantor measure is a spectral measure while the middle-third Cantor measure is not \cite{JP98}. This was the  first singular spectral measure without any point mass that was ever discovered. Such fractal self-similar measures were further studied  in the literature. In principle, fractal measures generated by infinite rescaled convolution of point mass spectral measures are essentially spectral.   For the details and recent advances, one can refer to \cite{AFL} and its references. 

\medskip

As one might expect, spectral measures are usually rare in nature due to the rigid orthogonality equation. People are looking for measures that admit exponential frames (also called Fourier frames) and exponential Riesz bases. One can show that sets of finite positive Lebesgue measures always  admit Fourier frames (for unbounded sets, see \cite{NOU}).  Sets admitting Riesz bases are much less understood. To the best of our knowledge, only finite unions of intervals and convex bodies with symmetric faces 
are known to admit Riesz bases of exponentials \cite{KN, KN16, DLev}.  Determining   whether many simple objects such as unit balls or triangles admit a Riesz basis is still an open problem.

\medskip

The existence of exponential Riesz bases and Fourier frames can be investigated in $L^2$ spaces of more general measures.  We call measures admitting Fourier frames or Riesz bases {\it frame-spectral measures} or {\it Riesz-spectral measures} respectively. The frequency set $\Lambda$ will be called a {\it frame-spectrum} or {\it Riesz-spectrum}.  General properties of these measures were given in \cite{HLL}. In \cite{LW,DEL}, fractal measures that are Riesz-spectral, but not spectral, were discovered. Strichartz \cite{St1}  asked whether the middle-third Cantor measure is frame-spectral, and the problem remains open as of today. Lev \cite{Lev} studied the  addition of two measures supported respectively on two orthogonal subspaces embedded in the ambient space ${\mathbb R}^d$. This additive measure will be singular with respect to the Lebesgue measure, and  he showed that these measures admit Fourier frames. This result was later generalized to the addition of finitely many restricted Lebesgue measures supported on union of subspaces. In particular, all polytope surface measures admit Fourier frames \cite{ILLW}.

\medskip

In this paper,  we continue the line of research into the spectrality of the addition of measures supported on two orthogonal subspaces. Our focus will be on $\R^2$,
with the measures supported on $\R^1$, though some of the results may be generalized to higher dimensions.
\medskip

We recall that a Borel measure $\mu$ is continuous if $\mu(\{x\}) = 0$ for all $x\in{\mathbb R}$. 

\vspace{1em}

\begin{definition}
Let $\mu$ and $\nu$ be  two  continuous Borel probability measures on $\R^1$. We 
  embed them into the $x$ and $y$ axes in $\R^2$ respectively.
     The \textit{additive space over $\mu$ and $\nu$}
  is the space $L^2(\rho)$, where $\rho$ is the measure
  \begin{equation}
    \label{rho_def}
    \rho = \frac12 {(\mu \times \delta_0 + \delta_0 \times \nu)} ,
  \end{equation}
  and $\delta_0$ is the Dirac measure at $0$. We will refer to the (compact) support of $\mu$ and $\nu$ as the {\it component spaces} of the measure $\rho$. If $\mu= \nu$, we say that $\rho$ is  {\it symmetric}. If $0\not\in$ supp$(\mu)$ $\cap$ supp$(\nu)$, we call $\rho$ {\it non-overlapping}. If $\mu,\nu$ are Lebesgue measures supported on intervals of length one, we call $\rho$ the {\it additive space of Lebesgue type.} (Here, supp$(\mu)$ denotes the compact support of $\mu$.)
\end{definition}

We are interested in \textit{spectral properties} of $L^2(\rho)$,  given that $\mu$ and $\nu$ are spectral (or Riesz-spectral, or frame-spectral). As mentioned above, frame-spectrality has already been studied by Lev in \cite{Lev}. So in this paper we mainly focus on Riesz bases and orthonormal bases. Our results show that the positions of the measures $\mu,\nu$ on $\R^1$ are important in determining the existence of a basis.

\medskip

\begin{theorem}\label{th_nonoverlapping}
Let $\rho$ be a non-overlapping symmetric additive measure with the component measure $\mu$. Suppose that $\mu$ is Riesz-spectral. Then so is $\rho$.
\end{theorem}

We will prove Theorem \ref{th_nonoverlapping} in Section 3. The overlapping case remains open, while we obtain the following necessary condition,
which helps in the proof of Theorem \ref{th_symmetric} below.

\medskip

\begin{theorem}
  \label{projections_of_Riesz_p}
  Suppose $E(\Lambda)$ is a Riesz basis for an additive
  space $L^2(\rho)$. Then for  at least one of the projections of $\Lambda$ onto the $x-$, $y-$axis, the corresponding collection of exponentials are not  Riesz bases for their 
  component spaces. 
\end{theorem}

The Fourier frames for the additive space constructed by Lev \cite[Theorem 1.1]{Lev}  was obtained by taking a subset of the cartesian product of the spectra of the component spaces. He actually deduced an algorithm for constructing this subset.   However, Theorem \ref{projections_of_Riesz_p} implies that one cannot generate a Riesz basis for the additive space by simply taking a subset of the cartesian product of the Riesz spectra from its component spaces. Furthermore, the frame spectra constructed by Lev \cite{Lev} cannot be a Riesz basis either (See Remark \ref{remark}). 

\medskip

We will prove Theorem \ref{projections_of_Riesz_p} in Section 5, by studying some limitations of zigzag paths in a Riesz basis for an additive space. We will show that the zigzags cannot be arbitrarily long.  However, if we assume both projections are Riesz bases for its component spaces, then arbitrarily long zigzags exist, which is a contradiction.

\medskip

We also use Theorem \ref{projections_of_Riesz_p} to deduce the following result about the spectrality of symmetric additive measures of Lebesgue type.

\medskip

\begin{theorem}\label{th_symmetric}
Let $\rho$ be a symmetric measure of Lebesgue type, where the component measure $\mu$ is the Lebesgue measure supported on $[t,t+1]$ and $-1/2<t \le 0$. 
\begin{enumerate}
\item If $t=0$, $\rho$ is spectral and has a unique spectrum up to translations.
\item If $2t+1 = \frac1{a}$ (i.e. $t = -\frac12+\frac1{2a}$), where $a>1$ is a positive integer, the measure is not spectral. 
\end{enumerate}
\end{theorem}

We will prove Theorem \ref{th_symmetric} in Section 6. 


\medskip

The paper proceeds as follows: Section 2 defines terms used in the rest of the paper.
We prove Theorem \ref{th_nonoverlapping} in Section 3. Section 4 contains theorems about the relationship between the (frame/Riesz-)spectra and the corresponding spectrality of its projections in additive spaces. Section 5 is devoted to proving Theorem \ref{projections_of_Riesz_p}.  Section 6 considers orthonormal bases in additive spaces, and contains our proof of Theorem \ref{th_symmetric}.   We end our paper with several interesting open problems in Section 7. 

\section{Preliminaries and definitions}

We will introduce some basic notation for our additive spaces in this section. 
 From (\ref{rho_def}),  any integrable function $F$ with respect to $\rho$ satisfies 
\vspace{1em}
\begin{equation}
  \label{int_def}
  \int F(x, y) \, d \rho (x,y) = 
  \frac 1 2 \int F(x, 0) \, d \mu(x) + \frac 1 2  \int F(0, y) \, d \nu(y).
\end{equation}
For a more compact form, if  $f(x,y)$ is a function of two variables,
  we define 
  \[
  f_x = f(x,0) \quad \text{ and } \quad f_y = f(0, y).
  \]
 Here $f_x$ is called the \textit{x projection} of $f$,
and $f_y$ is called the \textit{y projection} of $f$. Using this notation, Equation (\ref{int_def}) becomes
\[
  \int F \, d \rho = \frac 1 2 \int F_x \, d \mu + \frac 1 2 \int F_y \, d \nu.
\]

In $L^2(\rho)$, our inner product is 
\[
  \ang{f, g } = \int f \ov{g} \, d \rho.
\]
Thus 
\[
  \ang{f, g}_{L^2(\rho)}
  = \tfrac 1 2 \ang{f_x,g_x}_{L^2(\mu)} + \tfrac 1 2 
  \ang{f_y, g_y}_{L^2(\nu)}.
\]
We also have
\[
  \|f\|^2_{L^2(\rho)}= \tfrac 1 2 \|f_x\|^2_{L^2(\mu)} + \tfrac 1 2 
  \|f_y\|^2_{L^2(\nu)}.
\]

Usually
we write
inner products without
$L^2$ subscripts, expecting the
space to be clear from the context.
  \vspace{1em}
  

Here are some examples of additive spaces, along with the names
we have given them:

\begin{itemize}
  \item The \textit{L Space} (Figure \ref{fig_L})
has $[0,1]$ on the $x$-axis and $[0,1]$ on the
$y$-axis.

\item The \textit{Plus Space} (Figure \ref{fig_Plus})
has $[-1/2, 1/2]$ on each axis.

\item The \textit{T Space} (Figure \ref{fig_T})
has $[-1/2, 1/2]$ on the $x$-axis and $[-1,0]$ on the $y$-axis.

\item The \textit{symmetric spaces} have the same interval $[t,t+1]$ on
both 
axes.  If $-1\le t \le 0$, the symmetric space is {\it overlapping}. We only need to consider $t\in[-1/2,0]$ since reflections do not affect spectrality.   Both the L Space and the Plus Space are symmetric spaces.
Figure \ref{fig_Symm} has another example.
\end{itemize}

\begin{figure}[h]
  \centering
  \begin{minipage}[b]{.5 \textwidth}
    \centering
\begin{tikzpicture}[scale=.7]
  \draw[thick] (0,0) -- (3,0) node[anchor=west] {1};
  \draw[thick] (0,0) -- (0,3) node[anchor=south] {1};
  \filldraw[black] (0,0) circle (2pt) node[anchor=north east] {0};
\end{tikzpicture}
\captionof{figure}{The L Space}
  \label{fig_L}
\end{minipage}%
\begin{minipage}[b]{.5 \textwidth}
  \centering
\begin{tikzpicture}[scale=.7]
  \draw[thick] (-1.5,0) node[anchor=east] {$- \frac 1 2 $}  -- (1.5,0) 
  node[anchor=west] {$\frac 1 2$};
  \draw[thick] (0,-1.5) node[anchor=north] {$- \frac 1 2$} -- (0,1.5) 
  node[anchor=south] {$\frac 1 2$};
  \filldraw[black] (0,0) circle (2pt) node[anchor=north east] {0};
\end{tikzpicture}
\captionof{figure}{The Plus Space}
  \label{fig_Plus}
\end{minipage}

  \centering
  \begin{minipage}[b]{.5 \textwidth}
    \centering
\begin{tikzpicture}[scale=.7]
  \draw[thick] (-1.5,0) node[anchor=east] {$- \frac 1 2 $}  -- (1.5,0) 
  node[anchor=west] {$\frac 1 2$};
  \draw[thick] (0,-3) node[anchor=north] {$- 1$} -- (0,0);
  \filldraw[black] (0,0) circle (2pt) node[anchor=north east] {0};
\end{tikzpicture}
\captionof{figure}{The T Space}
  \label{fig_T}
\end{minipage}%
\begin{minipage}[b]{.5 \textwidth}
  \centering
\begin{tikzpicture}[scale=.7]
  \draw[thick] (-1,0) node[anchor=east] {$- \frac 1 3 $}  -- (2,0) 
  node[anchor=west] {$\frac 2 3$};
  \draw[thick] (0,-1) node[anchor=north] {$- \frac 1 3$} -- (0,2)
  node[anchor=south] {$\frac 2 3$};
  \filldraw[black] (0,0) circle (2pt) node[anchor=north east] {0};
\end{tikzpicture}
\captionof{figure}{A Symmetric Space}
  \label{fig_Symm}
\end{minipage}
\vspace{1em}
\end{figure}


\subsection{Exponential frames, Riesz bases, and orthonormal bases}

We study exponential functions 
on additive spaces. 
Here \textit{exponential} refers
to functions mapping $\R^2 \to \mathbb{T}\subset\C$ of the form
\[
(x,y) \to e^{2 \pi i (ax + by)},
\]
where $a, b \in \R$.  We use $e_{a,b}$ as a shorthand for this function;
similarly $e_a$ is the one-dimensional function
\[
x \to e^{2 \pi i a x}.
\]

\vspace{1em}
\begin{definition}
If $E$ is a set of exponential functions, then $\Lambda(E)$ (or
sometimes simply $\Lambda$) is the \textit{set of exponents} of $E$, that is
\[
    \Lambda = \{ (a,b) \mid e_{a,b} \in E \}.
\]
Similarly, $E$ is the \textit{set of exponentials} of $\Lambda$,
sometimes written as $E(\Lambda)$.
\end{definition}

With exponential functions, we have $e_{a,b}(x,0) = e_a(x)$
and $e_{a,b}(0,y) = e_b(y)$.  Hence
\[
  \ang{e_{a,b}, e_{u,v}}_{L^2(\rho)} = \tfrac 1 2 \ang{e_a, e_u}_{L^2(\mu)} + \tfrac 1 2
  \ang{e_b, e_v}_{L^2(\nu)}.
\]
Note also that for $e_{a,b}$, the $x$ projection
is $e_a$, and the $y$ projection is $e_b$. We will use these observations frequently in the paper.

\vspace{1em}


The exponential functions that we study are frames, Riesz bases, and orthonormal bases.
These are always 
exponential frames, exponential Riesz bases, and exponential orthonormal bases,
though for practical reasons we often omit the word \textit{exponential}.
The term \textit{spectral} usually
refers to a set or measure with an
exponential orthonormal basis.  Related terms are \textit{F-spectral}, referring
to a  set or measure with an
exponential frame, and \textit{R-spectral}, referring to a set or measure with an exponential Riesz basis.  The word \textit{spectral} in the title 
of our paper refers to all three types.

Below we define \textit{frames}, \textit{Riesz bases}, and \textit{orthonormal bases}
for an additive space $L^2(\rho)$. For the comprehensive theory of frames and bases, we refer  the reader to \cite{Chr,H}.   We assume that $\Lambda$ is
the (countable) set of exponents for the given set of functions.  Note
that an element of $\Lambda$ is a pair of real numbers.
\vspace{1em}

\begin{definition}
 $E(\Lambda)$ is a \textit{frame} for $L^2(\rho)$
  if there exist constants $A, B > 0$ such that the following formula,
  called the \textit{frame inequality}, holds:
  \begin{equation}
  \label{frame_def}
    \forall f \in L^2(\rho), \;\; A \| f \|^2 \le
    \sum_{(a,b) \in \Lambda} |\ang{f, e_{a,b}}|^2 \le B \| f \|^2.
  \end{equation}
  The constants $A$ and $B$ are the lower and upper \textit{frame bounds}.
  \end{definition}
  
  \vspace{1em}

\begin{definition}
$E(\Lambda)$
 is a \textit{Riesz basis} for $L^2(\rho)$ if
$\{ e_{a,b} \}$ is a frame, and if
  there exist constants
      $A, B > 0$ such that the following inequality,
      called the \textit{Riesz sequence inequality}, holds for any sequence
      of constants $\{ c_{a,b} \}$, $(a,b) \in \Lambda$, with only a finite
      number of nonzero values:
      \begin{equation}
      \label{Riesz_def}
        A \sum_{(a,b) \in \Lambda} |c_{a,b}|^2 \le
        \Bigl\| \sum_{(a,b) \in \Lambda}  c_{a,b} e_{a,b} \Bigr\|^2
        \le B \sum_{(a,b) \in \Lambda} | c_{a,b} |^2.
      \end{equation} \qedhere
      \end{definition}
      \vspace{1em}
      
 We can think of this inequality as saying that the norm of a finite linear
 combination of elements of the Riesz basis does not get too small or
 too large relative to its coefficients.
 
 \vspace{1em}
 
 \begin{definition}
$E(\Lambda)$ is an \textit{orthonormal basis} for $L^2(\rho)$ if
it is orthonormal, and satisfies \textit{Parseval's Equality}
\begin{equation}
\label{ONB_def}
  \forall f \in L^2(\rho), \;\|f\|^2 = \sum_{(a,b) \in \Lambda} | \ang{ f, e_{a,b}} |^2.
\end{equation} \qedhere
\end{definition}

We also note that $\Lambda-t$, $t\in{\mathbb R}^d$, are always frames/Riesz bases/orthonormal bases for $L^2(\rho)$ if the corresponding $\Lambda$ is. We may assume without loss of generality that $(0,0)\in\Lambda$.  This will be assumed throughout the paper.

\subsection{Projections}

Studying the projections of the spectrum onto axes will be our main tool for studying the spectrality of additive measure in later sections.  


%
%

\begin{definition}
  If $\Lambda \subset \R^2$ and $(a,b) \in \Lambda$,
    then the $x$ projection of $(a,b)$ is $a$
    and the $y$ projection is $b$.  The
    \textit{projection functions} $\pi_x$
    and $\pi_y$
    are defined as $\pi_x(a,b) = a$ and
    $\pi_y(a,b) = b$.
  $\Lambda_x$ is the set $\pi_x(\Lambda)$,
    and $\Lambda_y$ is $\pi_y(\Lambda)$.
  That is,
  \[
    \Lambda_x = \{ a \in \R : (a,b) \in \Lambda 
      \text{ for some } b \in \R \},
    \]
    and similarly for $\Lambda_y$.
\end{definition}
\vspace{1em}

Note that we freely apply $\pi_x$ and $\pi_y$
to both pairs and sets of pairs.
Also, if $\Lambda$ is a set of pairs,
and $a \in \Lambda_x$, then we write
$\pi_x^{-1}(a)$ to mean the set of pairs in $\Lambda$
having $a$ as their first member.

If $E(\Lambda)$ is a set of functions in $L^2(\rho)$,
then we can
consider
$E(\Lambda_x)$ and $E(\Lambda_y)$ to be functions in the component spaces $L^2(\mu)$ and
$L^2(\nu)$.  Thus it makes sense to ask, for example,
whether the projections of a frame in an additive space are also frames in the component spaces.

Projections of a linear combination of exponential functions
take an especially interesting form, used in later sections.
Let $f$ be a (finite or infinite) linear combination of the form
\[
  f = \sum_{(a,b) \in \Lambda} c_{a,b} e_{a,b},
\]
where each $c_{a,b} \in \C$.  One of its projections will be equal to 
\[
  f_x = \sum_{a \in \Lambda_x}
            \Bigl( \sum_{(a,b) \in \pi_x^{-1}(a)} c_{a,b}
            \Bigr) e_a.
          \]

Here we can think of the coefficients as being placed
on a two-dimensional grid, with $c_{a,b}$ appearing
at the point $(a,b)$ in the grid.  
The coefficient of $e_a$, namely
$            \sum_{(a,b) \in \pi_x^{-1}(a)} c_{a,b}$, 
is just the sum of all coefficients in one column
of this grid.

\medskip

%
%
%
%
%

\begin{definition}
  Let $\Lambda$ be the set of exponents for a set of exponential functions
  on an additive space.  
  The \textit{multiplicity} of $\Lambda$ is the largest number of points
  on any vertical or horizontal line through $\Lambda$, if such a maximum
  exists.
  We say that $\Lambda$ has $\textit{bounded multiplicity}$
  in this case.
  Similarly, the multiplicity of $u \in \Lambda_x$ (or $v \in \Lambda_y$) is the
  number of points on a vertical (or horizontal) line through $u$ (or $v$), if this number is finite.
\end{definition}

\section{Proof of Theorem \ref{th_nonoverlapping}}

Throughout the paper, we will write $A\asymp B$ if  there exist universal constants $c,C>0$ such that $cB\le A\le CB$.
We now study the non-overlapping case and prove Theorem \ref{th_nonoverlapping}. We will restate our theorem below with more detail about the structure of our Riesz basis. \vspace{1em}

\begin{theorem_th_nonoverlapping}
Suppose  that we have a non-overlapping symmetric additive space  whose measure is
$$
\rho = \frac12\left(\mu\times \delta_0 + \delta_0\times \mu\right),
$$
where $\mu$ is a compactly supported Borel probability measure on $\R^1$.
  Then if $L^2(\mu)$ admits a Riesz basis $E(\Lambda_{\mu})$, then $L^2(\rho)$ also admits a Riesz basis of the form 
  $$
\Lambda_{\rho} =  \{(\lambda,\lambda):  \lambda\in \Lambda_{\mu} \}+ \{(0,0), (0,\tau)\}
  $$ 
 for some $\tau\in{\mathbb R}\setminus\{0\}$. 
 \end{theorem_th_nonoverlapping}

\begin{proof}
We first show that $E(\Lambda_{\rho})$ is a frame for some $\tau\in{\mathbb R}$. Given any $F\in L^2(\rho)$,  we let $f = F(x,0)$ and $g = F(0,y)$ and $f,g\in L^2(\mu).$ 
$$
\begin{aligned}
4\sum_{(a, b)\in\Lambda_{\rho}} |\langle F,e_{a, b}\rangle|^2 =& \sum_{\lambda\in\Lambda_{\mu}} |\langle f,e_{\lambda}\rangle+\langle g,e_{\lambda}\rangle|^2+ \sum_{\lambda\in\Lambda_{\mu}} |\langle f,e_{\lambda}\rangle+\langle g,e_{\lambda+\tau}\rangle|^2\\
=&\sum_{\lambda\in\Lambda_{\mu}} |\langle f+g,e_{\lambda}\rangle|^2+ \sum_{\lambda\in\Lambda_{\mu}} |\langle f+g e_{-\tau},e_{\lambda}\rangle|^2\\
 \asymp& \int |f+g|^2 + |f+ge_{-\tau}|^2d\mu,
\end{aligned}
$$
 and the last estimate follows from the frame inequality of $E(\Lambda_{\mu})$ for $L^2(\mu)$. Let 
\begin{equation}\label{eqM}
{\bf M}_{\tau} = \begin{bmatrix} 1 & 1\\1 & e_{-\tau}\end{bmatrix}, \ {\bf v}(x) = \begin{bmatrix} f(x) \\ g(x) \end{bmatrix}.
\end{equation}
 Since $\mu$ is compactly supported with $0$ outside its support, we can find $\epsilon>0$ and  $\tau>0$ such that
 \begin{enumerate}
  \item $0<\epsilon<\tau\cdot x<1-\epsilon<1$ for all $x>0$ in the support of $\mu$, and 
  \item  $-1+\epsilon<\tau\cdot x<-\epsilon<0$ for all $x<0$ in the support of $\mu$.
  \end{enumerate} 
Indeed, if supp$(\mu)\subset [-M,-m]\cup[m,M]$, we can take  $\tau= \frac{1}{2M}$ and $\epsilon = \min(\frac12, \frac{m}{2M})$.   For this $\tau$, ${\bf M}$ is an invertible matrix and
 $$
\| {\bf M}_{\tau}{\bf v}\| \asymp \|{\bf v}\|,
 $$ 
where the implicit constant is equal to the largest and the smallest eigenvalue of ${\bf M}_{\tau}^{\ast} {\bf M}_{\tau}$. A calculation shows the eigenvalues are equal to 
$$
2 \pm|1+e_{-\tau}| = 2 \pm2\cos(\pi \tau x).
$$
In any case, all these eigenvalues are in the intervals $(2 (1-\cos (\pi \epsilon)), 4)$, which is uniformly away from zero. Hence, 
$$
\begin{aligned}
\int |f+g|^2 + |f+ge_{-\tau}|^2d\mu =& \int \|{\bf M}_{\tau}{\bf v}\|^2d\mu\\
 \asymp &\int \|{\bf v}\|^2d\mu \\
 =& \int |f|^2d\mu+\int |g|^2d\mu = 2\int |F|^2d\rho.\\
\end{aligned}
$$
This shows $E(\Lambda_{\rho})$ is a frame for some $\tau\in{\mathbb R}$. 

\vspace{1em}

To complete the proof, we need to show that $E(\Lambda_{\rho})$ is a Riesz sequence for the $\tau$ we specified in the previous paragraph. For any constants $c_{a,b}$, 
$(a,b)\in\Lambda_{\rho}$, with only finitely many non-zero,  we consider the sum
$$
2\int\left| \sum_{(a,b)\in\Lambda_{\rho}} c_{a,b}e_{a,b}\right|^2d\rho = \int\left| \sum_{\lambda\in\Lambda_{\mu}}c_{\lambda}e_{\lambda}+\sum_{\lambda\in\Lambda_{\mu}}c_{\lambda'}e_{\lambda}\right|^2d\mu+\int\left| \sum_{\lambda\in\Lambda_{\mu}}c_{\lambda}e_{\lambda}+\sum_{\lambda\in\Lambda_{\mu}}c_{\lambda'}e_{\lambda+\tau}\right|^2d\mu,
$$
where we denote $c_\lambda = c_{(\lambda,\lambda)}$ and $c_{\lambda'} = c_{(\lambda,\lambda+\tau)}$. Let $f(x) =\sum_{\lambda\in\Lambda_{\mu}}c_{\lambda}e_{\lambda} $ and $g(x)=\sum_{\lambda\in\Lambda_{\mu}}c_{\lambda'}e_{\lambda}$. Then the above expression is equal to 
$$
\int|f+g|^2+ |f+e_{\tau}g|^2d\mu = \int \|{\bf M}_{-\tau}{\bf v}\|^2d\mu.
$$
Using the same notation as in (\ref{eqM}), with the same argument,  we have 
$$
\begin{aligned}
 \int \|{\bf M}_{-\tau}{\bf v}\|^2d\mu\asymp&  \int \|{\bf v}\|^2d\mu \\
 = &\int |f|^2d\mu+\int |g|^2d\mu \\
 \asymp &\sum_{\lambda}|c_{\lambda}|^2+ \sum_{\lambda'}|c_{\lambda'}|^2 = \sum_{(a,b)\in\Lambda_{\rho}}|c_{a,b}|^2.\\
\end{aligned}
$$
The last $\asymp$ above follows from the fact that $E(\Lambda_{\mu})$ is a Riesz sequence of $L^2(\mu)$. 

\vspace{1em}

We have shown that $E(\Lambda_{\rho})$ is both a frame and a Riesz sequence. Hence it is a Riesz basis. The proof is complete. 
\end{proof}

We remark that the proof also holds if the component measures are different and zero is not in their supports, but they share the same Riesz basis spectrum.

\section{Frames, bases and their projections}

In this section, the component measures are all general continuous probability measures.  We will study some general properties of frames and Riesz bases of an additive space and the relationship of the frames and Riesz bases with their projections.  We will later use these theorems to construct an orthonormal basis or rule out the existence of an orthonormal basis for the additive spaces with Lebesgue component measures. 

\medskip

Our first theorem says that a frame in an additive space has bounded multiplicity.

\vspace{1em}

\begin{theorem}
  \label{bounded_multiplicity}
  Let $E(\Lambda)$ be a frame for an additive space with measure $\rho$ and  continuous component measures $\mu$ and $\nu$ defined in (\ref{rho_def}). Then
  \begin{enumerate}
  \item $\Lambda$ has bounded multiplicity.
  \item   $E(\Lambda_x)$ and $E(\Lambda_y)$
  are frames for $L^2(\mu)$ and $L^2(\nu)$ respectively. 
  \end{enumerate}
 \end{theorem}

\begin{proof}
  Let $\Lambda$ be a set of exponents for a frame for an additive space.
  Then there is some $A, B > 0$ 
  such that 
  for every $f \in L^2(\rho)$, the frame inequality holds (Equation \eqref{frame_def}):
  \[
    A \|f\|^2 \le
    \sum_{(a,b) \in \Lambda} | \ang{f, e_{a,b}} |^2 
    \le 
    B \|f\|^2.
  \]

  Fix $u \in \Lambda_x$, and choose $f$ to be
  $e_u$ on the $x$-axis and $0$ on the $y$-axis. (Its value at the origin is overspecified but irrelevant since the measures are continuous).
  That is, $f_x = e_u$ and $f_y = 0$.
      Then
  \[
    \|f\|^2 = \tfrac 1 2  \| e_u \|^2 + \tfrac 1 2 \| 0 \|^2 = \tfrac 1 2,
  \]
  and
  \[
    \ang{f, e_{a,b}} = \tfrac 1 2 \ang{e_u, e_a} + \tfrac 1 2 \ang{0, e_b}
    = \tfrac 1 2 \ang{e_u, e_a},
  \]
  and so from our frame inequality we have
  \[
    \sum_{(a,b) \in \Lambda} | \ang{e_u, e_a} |^2 
    \le B.
  \]
    Among the terms of this sum are all of
    the terms where $a = u$, and since every term
    is nonnegative, we have
    \[
    \sum_{(u,b) \in \Lambda} | \ang{e_u, e_u} |^2 
    \le B.
    \]
    But $\ang{e_u, e_u} = 1$, and so this inequality
    says that the multiplicity of $u$ is 
    bounded by $B$.
    A similar argument applies to
    $v \in \Lambda_y$. This proves part 1 of the theorem.
    
    \vspace{1em}
    
    To prove part 2 of the theorem, we let $M$ be the multiplicity of  $\Lambda$.  We show that $E(\Lambda_x)$ is a frame for $L^2(\mu)$.
  Let $g$ be any function in $L^2(\mu)$. Define $f \in L^2(\rho)$ such that
  $f_x = g$ and $f_y = 0$.
  Now
  \[
    \ang{f, e_{a,b}} = \tfrac 1 2 \ang{f_x, e_a} + \tfrac 1 2 \ang{f_y, e_b}
    = \tfrac 1 2  \ang{g, e_a} + 0,
  \]
  and
  \[
    \|f\|^2 = \tfrac 1 2 \|f_x\|^2 + \tfrac 1 2 \|f_y\|^2 = \tfrac 1 2 \| g \|^2 + 0.
  \]
  Our frame equation applied to $f$ becomes
  \[
    A \|g\|^2 \le 
     \sum_{(a,b) \in \Lambda} | \ang{g, e_a} |^2
    \le B \|g\|^2,
  \]
but 
  \[
   \sum_{(a,b) \in \Lambda} | \ang{g, e_a} |^2 =   \sum_{a \in \Lambda_x} m_a |\ang{g, e_a}|^2 ,
  \]
  where $m_a$ is the multiplicity of 
  $a \in \Lambda_x$. Since $1 \le m_a \le M$, this means
  \[
    \frac A M \|g\|^2 \le 
    \sum_{a \in \Lambda_x} |\ang{g, e_a}|^2 
    \le B \|g\|^2.
  \]
  Thus $E(\Lambda_x)$ is a frame
  of $L^2(\mu)$.
  The proof for $\Lambda_y$ is similar.
\end{proof}

\vspace{1em}

The previous theorem says that the projections of a frame are still frames in their component
spaces.  This is the  converse of Lev's Theorem 1.1 in \cite{Lev}, which provides a way to construct a frame for the additive space from
the frames of its component spaces.

\vspace{1em}

For the existence of a Riesz basis in an additive spaces and its relationship with its projections,  as we stated in Theorem \ref{projections_of_Riesz_p}, we know that if  $E(\Lambda)$ is a Riesz basis in an additive
  space, then at least one of  $E(\Lambda_x)$ and
  $E(\Lambda_y)$ is not  a Riesz basis for its
  component space. We give the proof in the next section.  
  
  \vspace{1em}
  
  Assuming Theorem \ref{projections_of_Riesz_p}, we now turn to study the orthonormal basis and its projection for additive spaces. As expected, they are more restrictive. 

\vspace{1em}
\begin{theorem}
  \label{multiplicity_one}
  Let $E(\Lambda)$ be an orthonormal basis for an additive space. Then
 \begin{enumerate} 
  \item $\Lambda$ has multiplicity one.
  \item Suppose that $\mu$ and $\nu$ are Lebesgue measures supported on intervals of length one. Then $\Lambda$ cannot be a subset of ${\mathbb Z}\times{\mathbb Z}$. 
  \end{enumerate}
   \end{theorem}

\begin{proof}
  Let $e_{a,b}$ and $e_{c,d}$ be two distinct members of $E$. For orthogonality, we must have
  \[
    \ang{e_{a,b}, e_{c,d}} = 0.
  \]
Letting $\lambda_1 = a - c$ and $\lambda_2 = b - d$, we have
\[
    \int e^{2 \pi i \lambda_1 x} \, d \mu (x) +
    \int e^{2 \pi i \lambda_2 y} \, d \nu (y) = 0.
\]
Suppose that we don't have multiplicity one. Then $\lambda_1=0, \lambda_2\neq 0$, or $\lambda_2=0, \lambda_1\neq 0$. In the first case,  we have
\[
  1 + \int e^{2 \pi i \lambda_2 y} \, d\nu( y) = 0.
\]
Equating the real parts this becomes $$\int (1+\cos (2\pi \lambda_2 y)) d\nu(y) = 0.$$ Notice the integrand is non-negative, so it follows that
$$
1+\cos (2\pi \lambda_2 y) = 0,\ \nu\text{-a.e}.
$$
However, this function equals zero if and only if $y = \frac{3}{4\lambda_2}+ \frac{k}{\lambda_2}$ where $k\in{\mathbb Z}$. This implies that $\nu$ is supported inside these countably many points. This contradicts the assumption that $\nu$ is a continuous measure. The case for $\lambda_2=0, \lambda_1\neq 0$ is similar.  We conclude that neither $\lambda_1$ nor $\lambda_2$ can be zero. Therefore $\Lambda$ has multiplicity one.

\vspace{1em}

We now prove part 2. Suppose $E(\Lambda)$ is an orthonormal basis, and $\Lambda \subseteq \Z \times \Z$.
Then $\Lambda_x \subseteq \Z$ and $\Lambda_y \subseteq \Z$.
However, Theorem \ref{bounded_multiplicity} shows that the projections of a frame
are frames in their component spaces, thus $E(\Lambda_x)$ and $E(\Lambda_y)$ are frames. Since
the set $\{ e^{2 \pi i n x} \}_{n \in \Z}$ 
is known to be an orthonormal basis for an interval of length one,
$\Lambda_x$ and $\Lambda_y$ cannot miss any integer.
That is, $\Lambda_x = \Z$ and $\Lambda_y = \Z$.
This means that $E(\Lambda)$ and its projections are all orthonormal bases.
This is impossible by Theorem 
\ref{projections_of_Riesz_p} because orthonormal bases are Riesz bases.  Hence $\Lambda$ cannot consist of only integer pairs.
\end{proof}

\section{Proof of Theorem \ref{projections_of_Riesz_p}} 
We first provide a heuristic proof. Suppose that $E(\Lambda)$ is a Riesz basis for $L^2(\rho)$ and both of its projections are Riesz bases. 
Let $(u,v)$ be any point in $\Lambda$.
We can consider a function $f\in L^2(\rho)$ such that $f_x = 0$ and $f_y = e_v$.  Then $f$ can be uniquely written as
$$
f = \sum_{(a,b)\in\Lambda} c_{a,b}\, e_{(a,b)}
$$
for some set of constants $\{ c_{a,b} \}_{(a,b) \in \Lambda}$.
Since $f=0$ on the $x$-axis, the vertical sums  $\sum_{(a,b)\in \pi_x^{-1} (x)}c_{a,b} = 0$ for all $x\in \Lambda_x$.
On the $y$-axis, where $f = e_v$, the horizontal sums $\sum_{(a,b)\in \pi_y^{-1} (y)}c_{a,b} = 0$ if $y\ne v$ and equal 1 if $y = v$, for all $y \in \Lambda_y$. Therefore 
$$
1=\sum_{y\in \Lambda_y}\sum_{(a,b)\in \pi_y^{-1} (y)}c_{a,b} =  \sum_{(a,b)\in\Lambda} c_{a,b} = \sum_{x\in \Lambda_x}\sum_{(a,b)\in \pi_x^{-1} (x)}c_{a,b}=0.
$$
This will give us a desired contradiction. However, this is not going to be valid since $c_{a,b}$ is only square-summable, and may not be absolutely summable. Therefore, the interchange of summation cannot be justified.  To overcome this issue, we will restrict our attention to finite sums, and consider the lengths of zigzag paths inside $\Lambda$. 

\subsection{Zigzags}

\begin{definition}
  Let $\Lambda \subset \R^2$.
  A \textit{zigzag} in $\Lambda$
is a finite or infinite sequence of points of the form
\[
(x_1, y_1),
(x_1, y_2),
(x_2, y_2),
(x_2, y_3),
\dots
\]
or of the form
\[
(x_1, y_1),
(x_2, y_1),
(x_2, y_2),
(x_3, y_2),
\dots
\]
A zigzag contains at least one point.
Each point in the zigzag is in $\Lambda$, and
no two consecutive points are identical.
The \textit{length} of a finite zigzag
is the number of points in the sequence minus one (i.e. number of paths inside the zigzag).
\end{definition}

Zigzags are horizontal and vertical. We think of zigzags as being composed of
zigs and zags, and we  will say  that a zig is horizontal and a zag is vertical.

A zigzag may begin with either a zig or a zag,
and a finite zigzag may end with either a zig or a zag.
Every zigzag alternates zigs and zags; this is required
by the definition when it says that no two
consecutive points are identical.

\vspace{1em}

\begin{definition}
  A \textit{zigzag loop} in $\Lambda$
  is a zigzag in $\Lambda$
  with positive even length
  whose first and last points are identical.
\end{definition}
\vspace{1em}

\begin{definition}
  Let $\Lambda \subset \R^2$.
  If $S \subset \Lambda$, 
  and for any $s \in S$, 
  there are no points in $\Lambda \setminus S$ that are
  on the same horizontal or vertical line 
  as $s$,
  then $S$ is said to be \textit{zigzag complete} in $\Lambda$.
\end{definition}
\vspace{1em}

If $S$ is zigzag complete, then
a zigzag in $\Lambda$ that contains any point
of $S$ must have all of its points in $S$.
That is, $S$ completely contains all of its zigzags.

Another fact will be important for us soon:
If $\Lambda \subset \R^2$,
and $\Lambda$ contains two different zigzags
from point $a \in \Lambda$ to point $b \in \Lambda$,
then $\Lambda$ contains a zigzag loop.


\subsection{ Riesz bases and zigzags}

This subsection collects the properties of zigzags inside the Riesz bases of an additive space $L^2(\rho)$.  In particular, we will prove that in a Riesz basis, there are no zigzag loops (Proposition \ref{no_loops}), and all zigzags have uniformly bounded length (Proposition \ref{zig_zag_chain}).  

\vspace{1em}

\begin{proposition}
  \label{no_loops}
  If $E(\Lambda)$ is a Riesz basis for an 
  additive space $L^2(\rho)$,
  then $\Lambda$ has no zigzag loop.
\end{proposition}

\begin{proof}
  Since $E(\Lambda)$ is a Riesz basis, we have the Riesz 
  sequence inequality (Equation \eqref{Riesz_def})
  \[
    C \sum_{(a,b) \in \Lambda} |c_{a,b}|^2
  \le
  \Bigl\| \sum_{(a,b) \in \Lambda} c_{a,b}
  e_{a,b} \Bigr\|^2
  \le 
  D \sum_{(a,b) \in \Lambda} |c_{a,b}|^2
\]
for some $C, D > 0$, where  only a finite number of $c_{a,b}$ will
be nonzero. We will show that the inequality cannot be true when
we assign constants $1$ and $-1$ in alternation to
the points of a zigzag loop.

  Suppose $\Lambda$ has a zigzag loop of the form
  \[
  (x_1,y_1),
  (x_1,y_2),
  (x_2,y_2),
  \dots,
  (x_n,y_n),
\]
where $(x_1,y_1) = (x_n,y_n)$.
Note that this begins with a vertical zag;
the proof for a loop beginning with a zig is similar. Drop $(x_n,y_n)$ from the sequence, and assign
constants to points in the loop by alternating
$1, -1, 1, -1, \dots$ as the loop is traversed.
Assign the constant $0$ to all other points of
$\Lambda$.
Note that the last point in the sequence is now
$(x_{n-1}, y_n)$, and $y_1 = y_n$ since the original
sequence was a loop.

Now the linear combination can be written
\[
  \sum_{(a,b) \in \Lambda} c_{a,b} e_{a,b} =
  e_{x_1,y_1} - e_{x_1,y_2} + e_{x_2,y_2} - \dots
  -e_{x_{n-2},y_{n-1}} + e_{x_{n-1},y_{n-1}} - e_{x_{n-1},y_n}.
\]
The $x$ projection is
\[
  (e_{x_1} - e_{x_1}) + (e_{x_2} - e_{x_2}) + \dots
  + (e_{x_{n-1}} - e_{x_{n-1}}) = 0
\]
and the $y$ projection is 
\[
  e_{y_1} + (- e_{y_2} + e_{y_2}) + \dots +
  (-e_{y_{n-1}} + e_{y_{n-1}}) - e_{y_n} = e_{y_1} - e_{y_n} = 0.
\]
See Figure \ref{fig_proj_zigzag_loop} for an example. Because of this cancellation,
\[
  \Bigl\| \sum_{(a,b) \in \Lambda} c_{a,b}
  e_{a,b} \Bigr\|^2 = \frac 1 2 \, | 0 |^2  + \frac 1 2 \, | 0 |^2  = 0,
\]
But this is a contradiction to the lower bound of the Riesz inequality since  $\sum_{(a,b) \in \Lambda} |c_{a,b}|^2>0$. Therefore $\Lambda$ cannot contain a zigzag loop.
\end{proof}

\vspace{1em}

\begin{figure}[h]
  \centering
  \begin{tikzpicture}[scale=.7]
  \draw[thick,->] (0,0) -- (0,4);
  \draw[thick,->] (0,0) -- (4,0);
  \draw[dotted] (0,1) -- (3,1);
  \draw[dotted] (0,3) -- (3,3);
  \draw[dotted] (1,0) -- (1,3);
  \draw[dotted] (3,0) -- (3,3);
  \filldraw[black] (0,1) circle (2pt) node[anchor=east] {$0e_b$};
  \filldraw[black] (0,3) circle (2pt) node[anchor=east] {$0e_c$};
  \filldraw[black] (1,0) circle (2pt) node[anchor=north] {$0e_a$};
  \filldraw[black] (3,0) circle (2pt) node[anchor=north] {$0e_d$};
  \filldraw[black] (1,1) circle (2pt) node[anchor=south west] {$1e_{a,b}$};
  \filldraw[black] (1,3) circle (2pt) node[anchor=south west] {$-1e_{a,c}$};
  \filldraw[black] (3,1) circle (2pt) node[anchor=south west] {$-1e_{d,b}$};
  \filldraw[black] (3,3) circle (2pt) node[anchor=south west] {$1e_{d,c}$};
\end{tikzpicture}
\caption{Projections of a Zigzag Loop\\
with Constants $1,-1,\dots$}
\label{fig_proj_zigzag_loop}
\end{figure}
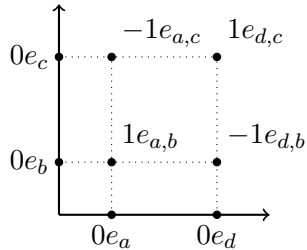

\vspace{1em}


\begin{proposition}
  \label{zig_zag_chain}
  If $E(\Lambda)$ is a Riesz basis for an additive
  space, then the zigzags in $\Lambda$ are uniformly
  bounded in length.  That is, 
  there exists a $B > 0$ such that
  no zigzag in $\Lambda$ is longer than $B$.
\end{proposition}

\begin{proof}
This proof is very similar to our proof that
Riesz bases contain no zigzag loops (Proposition \ref{no_loops}). We assume that we have
a zigzag of arbitrary length $N$.  We assign
$1, -1, 1, -1, \dots$ to the functions corresponding
to the zigzag, and $0$ to all other functions.
Again our goal is to force the $x$
and $y$ projections of the linear combination to be
``too small'', in the sense of violating the Riesz sequence inequality
\[
  \label{Riesz_seq_inequality}
    C \sum_{(a,b) \in \Lambda} |c_{a,b}|^2
  \le
  \Bigl\| \sum_{(a,b) \in \Lambda} c_{a,b}
  e_{a,b} \Bigr\|^2
  \le 
  D \sum_{(a,b) \in \Lambda} |c_{a,b}|^2.
\]
Again both the $x$ and $y$ projections have pairs
of functions of the form $e_{x_i} - e_{x_i}$
or $e_{y_i} - e_{y_i}$, which
cancel.

But the cancellation is not
total.  
Consider the zigzag in
Figure \ref{fig_proj_zigzag}.
As in the proof of Proposition \ref{no_loops}, we have
\[
  (e_{x_1} - e_{x_1}) + (e_{x_2} - e_{x_2}) + \dots
  + (e_{x_{n-1}} - e_{x_{n-1}}) = 0
\]
and
\[
  e_{y_1} + (- e_{y_2} + e_{y_2}) + \dots +
  (-e_{y_{n-1}} + e_{y_{n-1}}) - e_{y_n} = e_{y_1} - e_{y_n}.
\]
Here we don't know that $y_1 = y_n$, and so
we have
\[
  \Bigl\| \sum_{(a,b) \in \Lambda} c_{a,b}
  e_{a,b} \Bigr\|^2 = \frac 1 2 \, | 0 |^2 + \frac 1 2 \,
   | e_{y_1} - e_{y_n} |^2.
\]
Note that $e_{y_1} - e_{y_n}$ is a difference of two points on a
unit circle, and so the magnitude of this difference is
no more than 2.  Hence
\[
  \frac 1 2 \,| e_{y_1} - e_{y_n} |^2 \le 2.
\]

\begin{figure}
  \centering
\begin{tikzpicture}[scale=.6]
  \draw[thick,->] (0,0) -- (0,6);
  \draw[thick,->] (0,0) -- (5,0);
  \draw[dotted] (0,1) -- (1,1);
  \draw[dotted] (0,2) -- (2,2);
  \draw[dotted] (0,3) -- (3,3);
  \draw[dotted] (0,4) -- (4,4);
  \draw[dotted] (0,5) -- (4,5);
  \draw[dotted] (1,0) -- (1,2);
  \draw[dotted] (2,0) -- (2,3);
  \draw[dotted] (3,0) -- (3,4);
  \draw[dotted] (4,0) -- (4,5);
  \filldraw[black] (0,1) circle (2pt) node[anchor=east] {$1$};
  \filldraw[black] (0,2) circle (2pt) node[anchor=east] {$0$};
  \filldraw[black] (0,3) circle (2pt) node[anchor=east] {$0$};
  \filldraw[black] (0,4) circle (2pt) node[anchor=east] {$0$};
  \filldraw[black] (0,5) circle (2pt) node[anchor=east] {$-1$};
  \filldraw[black] (1,0) circle (2pt) node[anchor=north] {$0$};
  \filldraw[black] (2,0) circle (2pt) node[anchor=north] {$0$};
  \filldraw[black] (3,0) circle (2pt) node[anchor=north] {$0$};
  \filldraw[black] (4,0) circle (2pt) node[anchor=north] {$0$};
  \filldraw[black] (1,1) circle (2pt) node[anchor=west] {$1$};
  \filldraw[black] (1,2) circle (2pt) node[anchor=south] {$-1$};
  \filldraw[black] (2,2) circle (2pt) node[anchor=west] {$1$};
  \filldraw[black] (2,3) circle (2pt) node[anchor=south] {$-1$};
  \filldraw[black] (3,3) circle (2pt) node[anchor=west] {$1$};
  \filldraw[black] (3,4) circle (2pt) node[anchor=south] {$-1$};
  \filldraw[black] (4,4) circle (2pt) node[anchor=west] {$1$};
  \filldraw[black] (4,5) circle (2pt) node[anchor=south] {$-1$};
  \draw[thick] (1,1) -- (1,2) -- (2,2) -- (2,3) -- (3,3) -- (3,4) -- (4,4) -- (4,5);
\end{tikzpicture}
\caption{Projections of a Zigzag\\
with Constants $1,-1,\dots$}
\label{fig_proj_zigzag}
\end{figure}
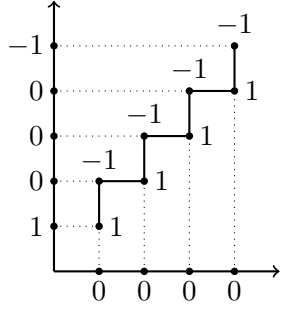
For other zigzags, the details vary,
but half of
the sum of the $x$
and $y$ projections is never more than 2.
So we have
\[
    C \sum_{(a,b) \in \Lambda} |c_{a,b}|^2
  \le
  \Bigl\| \sum_{(a,b) \in \Lambda} c_{a,b}
  e_{a,b} \Bigr\|^2
  \le 2.
\]
 We know $C > 0$, and
    $\sum_{(a,b) \in \Lambda} |c_{a,b}|^2$ is just
    the sum of the squares of all of our constants $1, -1, 1, -1, \dots$.
    For a zigzag of length $N$, this sum is $N+1$. Therefore, we have $(N+1)C\le 2$ for all $N\in{\mathbb N}$, which is a contradiction.
Therefore, the length of the zigzag
cannot be arbitrarily large.
\end{proof}

\begin{remark} \label{remark} {\rm We note that the frame constructed by Lev \cite[Theorem 1.1]{Lev}, regardless of the spectra chosen in the component space, does contain an infinitely long zigzag and is therefore not a Riesz basis. Indeed, in his construction, for each point $a\in \Lambda_x$, at least $q$ ($q\ge2$) points are constructed on the vertical line $x=a$. The same for each point in $\Lambda_y$. Hence, there is an infinite zigzag in the $\Lambda$ constructed.  }
\end{remark}

\subsection{Proof of Theorem \ref{projections_of_Riesz_p}}

\begin{theorem_projections_of_Riesz_p}
Suppose $E(\Lambda)$ is a Riesz basis for an additive
  space $L^2(\rho)$. Then at least one of $E(\Lambda_x)$ and $E(\Lambda_y)$ is not a Riesz basis for its component space.
 \end{theorem_projections_of_Riesz_p}

\begin{proof}

We argue by contradiction.
  Suppose that $E(\Lambda)$ is a Riesz basis, and 
  both $E(\Lambda_x)$ and $E(\Lambda_y)$ are
  Riesz bases of the component spaces.    
   We will show that $\Lambda$ has a
  zigzag chain of length $L$, where $L$ is any positive integer.
  This contradicts Proposition
  \ref{zig_zag_chain}, which says that a Riesz basis has
  uniformly bounded zigzags, and so 
  the projections cannot both be Riesz bases.
  
  Our proof is presented as a series of lemmas, stated and proved
  inside the current proof of the theorem.  A white box ($\square$)
  marks the end of a lemma proof. 
  
  \vspace{1em}

 \begin{lemma}
 \label{zigzag_complete}
    No finite subset of $\Lambda$ is zigzag complete.
  \end{lemma}

  \begin{lemmaproof}
  Let $F \subset \Lambda$ be a finite set that is zigzag complete.
  We derive a contradiction. Let $(u,v)$ be any point in $F$.
  Define a function $f$ in $L^2(\rho)$ so that
  \[
    f_x = 0 \qquad \text{and} \qquad
    f_y = e_v.
  \]
 Now, since $E(\Lambda)$ is a basis for $L^2(\rho)$, there is
  a unique set of coefficients $C = \{c_{a,b}\}_{(a,b) \in \Lambda}$
  specifying a linear combination
  of functions in $E(\Lambda)$ such that 
  \[
    f = \sum_{(a,b) \in \Lambda} c_{a,b} \, e_{a,b}.
  \]
  Since by hypothesis $E(\Lambda_x)$ is also a basis,
  the $x$ projection of $f$ is a unique linear combination
  \[
    f_x = \sum_{a \in \Lambda_x} 
    \Bigl( \sum_{(a,b) \in \pi_x^{-1}(a)} c_{a,b}\Bigr) \, e_a.
  \]
  However, by construction $f_x$ is the zero function,
  and so each coefficient in the above expansion is zero:
  \[
    \forall a \in \Lambda_x, \;
    \sum_{(a,b) \in \pi_x^{-1}(a)} c_{a,b} = 0.
  \]
  Thinking of the coefficients arranged in a two-dimensional
  plane, this tells us that every column sum of the set $C$
  is zero.  Note that these sums are finite and each has only finitely many terms, since $\Lambda$
  has bounded multiplicity (Theorem 
  \ref{bounded_multiplicity}).

  Similarly, $E(\Lambda_y)$ is also a basis, and so
  we have the unique linear combination
  \[
    f_y = \sum_{b \in \Lambda_y} 
    \Bigl( \sum_{(a,b) \in \pi_y^{-1}(b)} c_{a,b}\Bigr) \, e_b.
  \]
  But by construction 
\[
  f_y = e_v,
\]
and this equation is itself a linear 
  combination of elements in $E(\Lambda_y)$.
  Hence we have
  \[
    \sum_{(a,v) \in \pi_y^{-1}(v)} c_{a,v} = 1
  \]
  and for all $w \ne v$,
  \[
    \sum_{(a,w) \in \pi_y^{-1}(w)} c_{a,w} = 0.
  \]
  That is, every row sum of $C$ is zero,
  except for the row containing $c_{u,v}$,
  where the row sum is 1.
  
  Now consider the subset of $C$
  corresponding to points of the finite set $F$.
  Let
  \[
    C' = \{ c_{a,b} : (a,b) \in F \}.
  \]
  Since $F$ is zigzag complete, every column of $C$
  is either entirely within $C'$, or entirely outside $C'$.
  Similarly, every row of $C$ is either entirely
  within $C'$, or entirely outside $C'$.  In particular,
  the row whose sum is 1 is in $C'$, since $(u,v) \in F$.
  Thus $C'$ is a finite set whose sum is $0$ when added by
  columns and $1$ when added by rows.
  This is impossible.
  We conclude that
  $F$ is not zigzag complete, and thus no finite subset
  of $\Lambda$ is zigzag complete.
\end{lemmaproof}

\vspace{1em}


\begin{lemma}\label{lemma_S}
  There exists a sequence of distinct points 
  \[
  S = \{ s_n \}_{n = 1}^\infty, \; S \subset \Lambda,
  \]
  such that for every $s_k \in S$ with $k>1$, there exists an $s_j \in S$
  with $j < k$, such that either $s_j$ and $s_k$ are on the same
  horizontal line, or $s_j$ and $s_k$ are on the same vertical line.
\end{lemma}

  \begin{lemmaproof}
Choose an arbitrary point in $\Lambda$ to be $s_1$. We construct the sequence of points inductively. Suppose that we have picked $s_1,...,s_N$ such that the stated property holds for all $1<k\le N$. Then  the set $S_{N}=\{s_1,...,s_N\}$ is not zigzag complete by Proposition \ref{zigzag_complete}. There exists $s_j$ $(j\le N)$ such that we can find some $s_{N+1}\in \Lambda\setminus S_N$ such that $s_j,s_{N+1}$ lie either on the same horizontal line or the same vertical line. Hence,  such a sequence exists.
\end{lemmaproof}
\vspace{1em}

For the rest of the proof, we assume $S$ is a sequence
as described in the lemma above.
Also, we need a few definitions.
For any
point $s_k \in S$, let
\[
  V_k = \{ s_i \in S : 
  s_i \text{ and } s_k \text{ are on the same vertical line }
\}
\]
and
\[
  H_k = \{ s_i \in S : 
  s_i \text{ and } s_k \text{ are on the same horizontal line }
\}.
\]
We also need to designate the point with smallest index in 
$V_k$ or $H_k$.  We call these $s_{v(k)}$ and $s_{h(k)}$,
respectively.
That is, for any $k$ we have
$s_{v(k)} \in V_k$, and for any $s_i \in V_k$ we have
\begin{equation} \label{eqv_k}
v(k) \le i,
\end{equation}
and similarly for $s_{h(k)}$ and $H_k$.  In particular, $v(k) \le k$ and $h(k) \le k$.

  \vspace{1em}
\begin{lemma}
  Let $f: S \to \Z$ be the function defined recursively as follows:
  \[
    f(s_n) = \begin{cases}
      0 & n = 1,\\
      f(s_{v(n)}) + 1 & v(n) < n,\\
      f(s_{h(n)}) + 1 & h(n) < n.
    \end{cases}
  \]
  Then $f$ is well defined, and for any $s_k \in S$,
  $f(s_k)$ is the length of a zigzag in $S$
  (hence in $\Lambda$)
  from $s_1$ to $s_k$.
\end{lemma}
Since there is no zigzag loop in $\Lambda$, the zigzag from $s_1$ to $s_k$ is unique and therefore $f(s_k)$ is the zigzag distance from $s_1$ to $s_k$.

\begin{lemmaproof}
We argue by induction
on $S$.  

Consider $s_1$ first.  Since 
in the definition of $f$ we have $n = 1$ and
$v(1) = h(1) = 1$, only
the first clause in the definition will apply, and
so $f(s_1) = 0$.
Thus, $f$ is well-defined on $s_1$, and $f(s_1)$ is the
length of a zigzag from $s_1$ to $s_1$.

Now consider $s_n$, for $n > 1$.  Assume that for 
all $k < n$, $f(s_k)$ is well-defined 
and is the length of a zigzag in $S$ from $s_1$ to $s_k$.
We further assume that the zigzag goes from points
of lower index in $S$ to points of higher index.
(Note that this is true for the base case above,
since that zigzag has length 0.)

We show that $f(s_n)$ is also well-defined, and is the
length of a zigzag from $s_1$ to $s_n$
that goes from points of lower index to points of
higher index.

First we show that $f(s_n)$ is well-defined.  The first condition 
in the function definition does not apply since $n > 1$. Suppose $f(s_n)$ is not well-defined. Then we have two cases:
\begin{enumerate}[(1)]
\item $h(n) \ge  n$ and $v(n) \ge n$;
\item $v(n) < n$ and $h(n) < n$.
\end{enumerate}

Case (1) is not possible because then we would have $h(n) = v(n)=n$ (by (\ref{eqv_k})).  This cannot be
true except for $n=1$, since the definition of $S$ in Lemma \ref{lemma_S} says that
every point after $s_1$ 
must be on the same horizontal or vertical line as some
predecessor in $S$. 

Case (2) is not possible either. First notice that in this case $h(n) \ne v(n)$, since $s_n$ is the only intersection point between the horizontal and vertical lines passing through itself.
Therefore $s_{h(n)} \ne s_{v(n)}$. By the induction hypothesis there is
a zigzag from $s_1$ to $s_{v(n)}$ and another from
$s_1$ to $s_{h(n)}$.
This situation cannot arise,
because we could then construct two different zigzags from $s_1$
to $s_n$, one ending with a horizontal zig from $s_{h(n)}$
and one ending with a vertical zag from $s_{v(n)}$.
We would have a zigzag loop, which is
not possible in a Riesz basis (Proposition \ref{no_loops}). 

Hence either $h(n)<n$ or $v(n)<n$ holds, but not both.  We conclude that $f$ is well-defined by the definition
above.  

Now we show that $f(s_n)$ is the length of a zigzag
from $s_1$ to $s_n$.  We consider the case where $v(n) < n$;
the other case ($h(n) < n$) is similar.
Since $v(n) < n$, there is a zigzag from $s_1$ to $s_{v(n)}$,
by our induction hypothesis.
The last component of this zigzag is not a vertical zag,
because the zigs and zags go from points of lower index
in $S$ to points of higher index, and by definition $s_{v(n)}$ is
the point of lowest index in $V_n$.  Thus, the zigzag from
$s_1$ to $s_{v(n)}$ is either a zigzag of zero length 
(if $v(n) = 1$), or it is a zigzag that ends with a horizontal
zig.  In either case, by adding a vertical zag from 
$s_{v(n)}$ to $s_n$, we make a zigzag from $s_1$ to
$s_n$.
Our last zag goes from a
point of lower index to one of higher index in $S$,
and the length of the entire zigzag is $f(s_{v(n)}) + 1$,
which matches our definition of $f$.
Therefore $f(s_n)$ is the length of a zigzag from $s_1$ to $s_n$.
\end{lemmaproof}
\vspace{1em}

Next we show that for any $N \in \Z$, the number of points
$s_k \in S$ with $f(s_k) = N$ is finite.

  \vspace{1em}
\begin{lemma}
  For any $N \in \Z$, $f^{-1}(N)$ is a finite set.
  That is, $f$ takes on any value $N$ only a finite number
  of times.
\end{lemma}

  \begin{lemmaproof}
Our proof is by induction on $N$.  For the base
case of $N = 0$, we have $f^{-1}(0) = \{ s_1 \}$, which
is a finite set.  

Now assume that $f^{-1}(k)$ is finite for each $k < N$.
We show $f^{-1}(N)$ is also finite.
To do this, we consider all of the ways that a point
$s_n$ can be assigned $f(s_n) = N$.  
Looking at the definition of $f$, we see that this can
occur for any point $s_n$ where $v(n) < n$,
as long as $s_{v(n)}$ is assigned $N-1$.
But the number of points where $f$ is $N-1$ is finite,
by our induction hypothesis.
Also, $\Lambda$ has bounded multiplicity (Theorem
\ref{bounded_multiplicity}),
and so the number of points in each
$V_k$ is finite.  Hence the total number of points $s_n$
assigned $f(s_n) = N$ in the case $v(n) < n$ is finite.
A similar argument applies to points satisfying
$h(n) < n$.
Therefore $f^{-1}(N)$ is finite.
\end{lemmaproof}
\vspace{1em}

Let us summarize our progress so far.
We have a sequence $S \subset \Lambda$, 
where each point in $S$ other than the first
lies on the same horizontal line or the same vertical line
as some predecessor.  We are able to assign
to each point $s_n$ of $S$ a non-negative integer $f(s_n)$ giving the
zigzag distance in $S$ from $s_1$ to $s_n$.  
Also, we showed that for any possible distance $N$, there are
only a finite number of points in $S$ at distance $N$ from $s_1$.

  \vspace{1em}
\begin{lemma}
  For any positive integer $L$,
  there is a point $s_n \in S$ such that $f(s_n) \ge L$.
\end{lemma}

  \begin{lemmaproof}
Since the number of points of $S$ at each possible distance 
from $s_1$ is finite,
the total number of points at distance less than $L$ is also
finite.  Therefore, there is some $s_n$ 
such that $f(s_n) \ge L$.
\end{lemmaproof}
\vspace{1em}

And now we conclude the proof of our theorem.
  We have shown that $\Lambda$ has a zigzag of length $L$, where
  $L$ is arbitrary.  This violates Proposition
  \ref{zig_zag_chain}, saying that
  Riesz bases have uniformly bounded zigzags.
  Therefore, it is not possible for
  both $\Lambda_x$ and $\Lambda_y$ to be sets of exponents
  of Riesz bases. The proof is now complete. 
  \end{proof}

\vspace{1em}

\section{Finding orthonormal bases}

In this section, we will investigate which additive spaces of Lebesgue measure on intervals will admit an exponential orthonormal basis. We first note that the orthonormality gives rise to an important orthogonality equation.

\vspace{1em}

\begin{lemma}
Suppose $E(\Lambda)$ is an orthonormal basis for the additive space whose components
are $[t, t+1]$ on the $x$ axis and $[t', t'+1]$ on the $y$ axis.  Let $(a,b)$ and $(c,d)$ be any
two distinct points of $\Lambda$.  Let $(\lambda_1, \lambda_2) = (a,b) - (c,d)$.
Then $\lambda_1$ and $\lambda_2$ satisfy $\lambda_1\lambda_2\neq 0$ and
\begin{equation}
  \label{orthogonality_equation}
  e^{\pi i \lambda_1 (2t+1)} \; \frac {\sin(\pi \lambda_1)}
  {\pi \lambda_1} =
  - e^{\pi i \lambda_2 (2t'+1)} \; \frac {\sin(\pi \lambda_2)}
  {\pi \lambda_2}.
\end{equation}
\end{lemma}

\begin{proof}
Since
  \[
    \ang{e_{a,b}, e_{c,d}} = 0,
  \]
we must have
  \[
    \int e^{2 \pi i \lambda_1 x} \, d \mu +
    \int e^{2 \pi i \lambda_2 y} \, d \nu = 0.
\]
Since $\Lambda$ has multiplicity one (Theorem \ref{multiplicity_one}),
neither $\lambda_1$ nor $\lambda_2$ is zero.
Because
\begin{align*}
  \int_t^{t+1} e^{2 \pi i \lambda y} \, d y
  &= e^{\pi i \lambda (2t+1)} \; \frac {\sin(\pi \lambda) }
  {\pi \lambda},
\end{align*}
we have
\[
  e^{\pi i \lambda_1 (2t+1)} \; \frac {\sin(\pi \lambda_1)}
  {\pi \lambda_1} =
  - e^{\pi i \lambda_2 (2t'+1)} \; \frac {\sin(\pi \lambda_2)}
  {\pi \lambda_2}.
  \]
\end{proof}

We call \eqref{orthogonality_equation} the \textit{Orthogonality Equation}.  
Note that the Orthogonality Equation is slightly different for each additive space
(i.e., for each choice of $t$ and $t'$).

\subsection{Proof of Theorem \ref{th_symmetric}, part 1}
\label{section_L}

Here we show that the L Space has a unique orthonormal basis.

\vspace{1em}

\begin{theorem_th_symmetric1}
Assuming $(0,0)\in\Lambda$, the only orthonormal basis for the L Space is $E(\Lambda)$, where
$$
\Lambda = \{ (n/2, -n/2) \mid n \in \Z \}.
$$
\end{theorem_th_symmetric1}

\begin{proof}
Let $E(\Lambda)$ be a prospective orthonormal basis for the L Space.
Our proof is in three steps:
\begin{enumerate}
\item Solve the Orthogonality Equation (Equation \eqref{orthogonality_equation}) for the L Space to find the possible values of $\lambda_1$ and $\lambda_2$,
the differences between points of $\Lambda$.
\item Use this solution to show that $\Lambda$ must have the form given in the theorem statement.
\item Show that $E(\Lambda)$ is in fact an orthonormal basis for the L Space.
\end{enumerate}

\textbf{Step 1}: We solve the Orthogonality Equation to find values for the differences in $\Lambda$.

The Orthogonality Equation for the L Space (where $t = t' = 0$ in terms of Equation \eqref{orthogonality_equation})
is
  \begin{equation}
    \label{OE_L}
e^{\pi i \lambda_1} \; \frac {\sin(\pi \lambda_1)} {\pi \lambda_1} = - e^{\pi i \lambda_2} \; \frac {\sin(\pi \lambda_2)} {\pi \lambda_2},\ \lambda_1\lambda_2\neq 0. 
  \end{equation}
We will show that this equation implies that either
\begin{enumerate}
\item $\lambda_1, \lambda_2 \in \Z$, or
\item $\lambda_1 = 1/2 + n$ for some $n \in \Z$, and $\lambda_2 = - \lambda_1$.
\end{enumerate}

  Looking at \eqref{OE_L}, we note that
$e^{\pi i \lambda_1}$ and $ - e^{\pi i \lambda_2}$ 
are points on the unit circle in $\C$, and 
${\sin(\pi \lambda_1)} / {\pi \lambda_1}$ and 
${\sin(\pi \lambda_2)} / {\pi \lambda_2}$ are real-valued multipliers.
Therefore, there are three ways that \eqref{OE_L} might be true.

\textit{Case 1}: The multipliers 
${\sin(\pi \lambda_1)} / {\pi \lambda_1}$ and ${\sin(\pi \lambda_2)}  / {\pi \lambda_2}$
are $0$.  Then $\sin(\pi \lambda_1) = \sin(\pi \lambda_2) = 0$,
and so $\lambda_1$ and $\lambda_2$ are nonzero integers.

\textit{Case 2}: The multipliers 
${\sin(\pi \lambda_1)} / {\pi \lambda_1}$ and ${\sin(\pi \lambda_2)}  / {\pi \lambda_2}$
are nonzero and equal.
In this case, from \eqref{OE_L} we have
$$
e^{\pi i \lambda_1} = - e^{\pi i \lambda_2}.
$$
Using $e^{i \theta} = \cos(\theta) + i \sin(\theta)$, 
\begin{align*}
  \cos \pi \lambda_1 &= - \cos \pi \lambda_2\\
  \sin \pi \lambda_1 &= - \sin \pi \lambda_2.
\end{align*}
Since the multipliers
${\sin(\pi \lambda_1)} / {\pi \lambda_1}$ and ${\sin(\pi \lambda_2)}  / {\pi \lambda_2}$
are equal, but
$\sin \pi \lambda_1 = - \sin \pi \lambda_2$,
we must have 
$\lambda_1 = - \lambda_2$.
Then
$$
\cos \pi \lambda_1 = - \cos \pi \lambda_1,
$$
or $\cos \pi \lambda_1 = 0$.
Thus $\lambda_1$ is $1/2 + n$, for some $n \in \Z$, and $\lambda_2 = - \lambda_1$.

\textit{Case 3}: The multipliers 
${\sin(\pi \lambda_1)} / {\pi \lambda_1}$ and ${\sin(\pi \lambda_2)}  / {\pi \lambda_2}$
are nonzero and of opposite sign.
Then, from \eqref{OE_L}
$$
e^{\pi i \lambda_1} =  e^{\pi i \lambda_2}.
$$
Using $e^{i \theta} = \cos(\theta) + i \sin(\theta)$ again,
\begin{align*}
\cos \pi \lambda_1 &=  \cos \pi \lambda_2\\
\sin \pi \lambda_1 &=  \sin \pi \lambda_2.
\end{align*}
Now ${\sin(\pi \lambda_1)} / {\pi \lambda_1} = - {\sin(\pi \lambda_2)}  
/ {\pi \lambda_2}$ (because these are the multipliers
of opposite sign),
but $
\sin \pi \lambda_1 =  \sin \pi \lambda_2$.
Then $\lambda_1 = - \lambda_2$, and so
$$
\sin \pi \lambda_1 = - \sin \pi \lambda_1,
$$
or $\sin \pi \lambda_1 = 0$.  This is impossible when the multipliers are nonzero.

We conclude that either $\lambda_1, \lambda_2 \in \Z$
or $(\lambda_1, \lambda_2) = (1/2 + n, - 1/2 - n)$ for some $n \in \Z$.

\textbf{Step 2}: Given these values for the differences between points of $\Lambda$,
we determine the elements of $\Lambda$.

We have shown that 
if $\lambda, \lambda' \in \Lambda$, then $\lambda - \lambda'$
is either a pair of integers, or $(1/2 + n, -1/2 - n)$ for some $n \in \Z$.
Note that in the latter case, we have $\lambda_1 = - \lambda_2$.  
Of course, we could have a mixture of these two
types of differences within $\Lambda$.

Define the relation
$\lambda \sim \lambda'$ to mean that $\lambda - \lambda'$ is a pair of integers.
Then $\sim$ is an equivalence relation on $\Lambda$,
dividing $\Lambda$ into two equivalence classes $P_1$ and $P_2$.
Assign the names $P_1$ and $P_2$
so that $(0,0) \in P_1$.
Then
$P_1$ is nonempty and contains the integer members of $\Lambda$.  
$P_2$ contains members whose components are odd multiples of $1/2$.
$P_2$ is also nonempty, since 
$\Lambda$ cannot consist only of integers (Theorem \ref{multiplicity_one}).

Choose $(a,b) \in P_1$, and $(u,v) \in P_2$.  
Since $(a,b) - (u,v)$ is not a pair of integers, we must have 
$a - u = -(b - v)$, or
$$
a + b = u + v.
$$
Since we can let $(u,v)$ range over $P_2$, and $(a,b)$ range over $P_1$, this 
means
that every pair in $\Lambda$ has the same sum.  In particular, since $(0,0) \in \Lambda$,
that sum is $0$.  Then $\Lambda$ contains some elements of the form $(n, -n)$, for $n \in \Z$, and some elements
of the form $(1/2 + n, -1/2 - n)$, for $n \in \Z$.  
We write this more compactly as
$$
\Lambda \subseteq \{ (n/2, -n/2) \mid n \in \Z \}.
$$
But in fact 
$\Lambda$ must be equal to this set, since the set on the right hand side is
orthogonal,
as demonstrated shortly.  
(If any point of this form were missing from $\Lambda$,
the corresponding function would be orthonormal to every member of $E(\Lambda)$.)
Therefore
$$
\Lambda = \{ (n/2, -n/2) \mid n \in \Z \}.
$$

\textbf{Step 3}: We show that $E(\Lambda)$ is in fact an orthonormal basis.

First we show that $E(\Lambda)$ is an orthonormal sequence.
Suppose we have two different pairs of $\Lambda$, which we
can call $(h_1, -h_1)$ and $(h_2, -h_2)$.
Here $h_1 = n_1/2$ and $h_2 = n_2/2$ for some $n_1, n_2 \in \Z$
with $n_1 \ne n_2$.
Then
\begin{align*}
  \ang{e_{h_1, -h_1},
  e_{h_2, -h_2}} &=
  \int_0^1 e^{2 \pi i \left( h_1 - h_2 \right) x} \,dx
  + \int_0^1 e^{2 \pi i \left( h_2 - h_1 \right) x} \,dx\\[1em]
  &= \frac{e^{\pi i (n_1 - n_2)} - e^{\pi i (n_2 - n_1)}}
  {\pi i (n_1 - n_2)}.
\end{align*}
Since $n_1$ and $n_2$ are integers, $n_1 - n_2$ and $n_2 - n_1$
are integers with opposite signs.  Therefore the above inner product is zero 
and so the set $E(\Lambda)$ is indeed orthonormal.

Now, we show that Parseval's Equality is true, namely
\[
  \|F\|^2 = \sum_{\lambda \in \Lambda} | \ang{ F, e_{\lambda}} |^2
\]
for any $F \in L^2(\rho)$.
Let $f(x) = F(x, 0)$ and
$g(y) = F(0, y)$.
Let $e_\lambda \in \Lambda$, so that $\lambda = (n/2, -n/2)$
for some $n \in \Z$.

We begin by looking at $\ang{ F, e_{\lambda}}$. By a change of variable, 
\begin{align*}
\ang{ F, e_{\lambda}} = \frac{1}{2}
\int_0^1 f(x) e^{- \pi i n x} \, dx 
+ \frac{1}{2}\int_{-1}^0 g(-x) e^{- \pi i n x} \, dx .
\end{align*}
Define a function $H: [-1,1] \to \C$ as
\[
  H(x) = \begin{cases}
    g(-x) & x \in [-1, 0]\\
    f(x) & x \in [0,1].
  \end{cases}
\]
Then
\[
  \ang{ F, e_{\lambda}} = \frac{1}{2}\int_{-1}^1 H(x) e^{-\pi i n x} \,dx.
\]

Since $\{ e^{2 \pi i n x} \}_{n \in \Z}$ is an orthonormal basis
for $L^2([0,1])$, we know that 
$\{\frac{1}{\sqrt{2}} e^{\pi i n x} \}_{n \in \Z}$ is an orthonormal basis
for $L^2([-1,1])$.  Thus we can apply Parseval's Equality to
$H$ itself, on the interval $[-1,1]$.  We have

\begingroup
\begin{align*}
  \sum_{\lambda \in \Lambda} | \ang{ F, e_{\lambda}} |^2
  = &\frac{1}{2}\int_{-1}^1 | H(x) |^2 \, dx\\
  =& \frac{1}{2}\int_0^1 |f(x)|^2 \, dx + \frac{1}{2}\int_0^1 |g(x)|^2 \, dx\\
  = &\int_0^1 |F(x,y)|^2 \, d \rho(x,y) \\
\end{align*}
\endgroup
We conclude that $E(\Lambda)$ is an orthonormal basis for the L Space. 
But we have also shown that
any orthonormal basis
must be of this form.  Therefore, $E(\Lambda)$ is the only
orthonormal basis for the L Space.
\end{proof}

Figure \ref{fig_L_basis} shows the basis and the space together.
\vspace{1em}

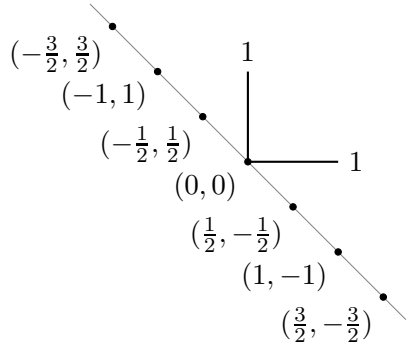
\begin{figure}[h]
  \centering
\begin{tikzpicture}[scale=.6]
  \draw[thick] (0,0) -- (0,2) node[anchor=south] {$1$};
  \draw[thick] (0,0) -- (2,0) node[anchor=west] {$1$};
  \draw[gray,very thin]  (-3.5,3.5) -- (3.5,-3.5);
  \filldraw[black] (-3,3) circle (2pt) node[anchor=north east] 
    {$(-\frac 3 2, \frac 3 2)$};
  \filldraw[black] (-2,2) circle (2pt) node[anchor=north east] {$(-1,1)$};
  \filldraw[black] (-1,1) circle (2pt) node[anchor=north east] 
    {$(-\frac 1 2, \frac 1 2)$};
  \filldraw[black] (0,0) circle (2pt) node[anchor=north east] {$(0,0)$};
  \filldraw[black] (1,-1) circle (2pt) node[anchor=north east] 
    {$(\frac 1 2, -\frac 1 2)$};
  \filldraw[black] (2,-2) circle (2pt) node[anchor=north east] {$(1, -1)$};
  \filldraw[black] (3,-3) circle (2pt) node[anchor=north east] 
    {$(\frac 3 2, -\frac 3 2)$};
\end{tikzpicture}
\caption{The L Space and its Orthonormal Basis}
\label{fig_L_basis}
\end{figure}

\FloatBarrier

\begin{remark}
{\rm From Step 3 in the proof, we can also obtain an orthonormal basis for the additive space where the components are $[k,k+1]$ and $k = 1,2,3...$. Indeed, the spectrum is }
$$
\Lambda = \{(\lambda,-\lambda): \lambda\in\Lambda_0\}
$$
{\rm where $\Lambda_0$ is a spectrum for $\Omega = [-(k+1),-k]\cup[k,k+1]$. Such a spectrum exists since $\Omega$ is a translational tile of two disjoint intervals,
and Fuglede's conjecture holds true for union of two intervals \cite{Laba2001}. We don't know however if $\Lambda$ is the unique spectrum up to translations. }
\end{remark}

\subsection{Proof of Theorem \ref{th_symmetric}, part 2}

Symmetric spaces are ones where the $x$ and $y$ components
are positioned at the same location.  We have already shown that at least one symmetric space,
the L Space,
has an orthonormal basis.  Here we find an
infinite class of symmetric spaces that have no orthonormal basis. 


\vspace{1em}
\begin{theorem_th_symmetric2}
Suppose a symmetric additive space has components $[t, t+1]$,
and $2t+1 = 1/a$ for some integer
$a > 1$.  Then there is no orthonormal basis for this space.
\end{theorem_th_symmetric2}

\begin{proof}
  In this case,
  the Orthogonality Equation 
  \eqref{orthogonality_equation}
  becomes
\[
  e^{\pi i \lambda_1 (1/a)} \; \frac {\sin(\pi \lambda_1)}
  {\pi \lambda_1} =
  - e^{\pi i \lambda_2 (1/a)} \; \frac {\sin(\pi \lambda_2)}
  {\pi \lambda_2},\ \lambda_1\lambda_2\neq 0.
\]

The proof is similar to our analysis of the L Space.
The Orthogonality Equation can be true in three cases.

\textit{Case 1}: The real multipliers 
  $\sin(\pi \lambda_1)/\pi \lambda_1$ 
  and
  $\sin(\pi \lambda_2)/\pi \lambda_2$ 
are zero.  In this case
$\lambda_1$ and $\lambda_2$ are nonzero integers.

\textit{Case 2}: The multipliers are nonzero and equal.
In this case
\[
  e^{\pi i \lambda_1 (1/a)}  =
  - e^{\pi i \lambda_2 (1/a)},
\]
and so 
\[
  \frac 1 a \pi \lambda_1 + \pi + 2n \pi = \frac 1 a \pi \lambda_2,
\]
for some $n \in \Z$.
That is,
\[
  \lambda_1 + a(2n + 1) = \lambda_2.
\]

If $a$ is even, then $a(2n+1)$ is an even integer, and so $\sin(\pi \lambda_1)
= \sin(\pi \lambda_2)$, implying $\lambda_1 = \lambda_2$
(since the multipliers are equal), which is impossible
since $a(2n+1) \ne 0$.

If $a$ is odd, then $a(2n+1)$ is an odd integer, and so
$\sin(\pi \lambda_1)
= - \sin(\pi \lambda_2)$, implying $\lambda_1 = - \lambda_2$.
Thus $\lambda_1 = (2n+1) a/2$, that is, $\lambda_1$ and $\lambda_2$ are
odd multiples of $a/2$, and of opposite sign.

\textit{Case 3}:
The multipliers are nonzero and opposite in sign.
Then
\[
  e^{\pi i \lambda_1 (1/a)}  =
  e^{\pi i \lambda_2 (1/a)},
\]
and so 
\[
  \lambda_1 + a(2n) = \lambda_2
\]
for some $n \in \Z$.
Since $a(2n)$ is an even integer, we have $\sin(\pi \lambda_1)
= \sin(\pi \lambda_2)$, implying $\lambda_1 = -\lambda_2$
(since the multipliers are of opposite sign).
Then
\[
  \lambda_1 + a(2n) = - \lambda_1,
\]
and so $\lambda_1$ and $\lambda_2$ are nonzero integers,
which is not possible if the multipliers are nonzero.

To summarize our findings: if $a$ is even, the Orthogonality Equation
has no solutions except $\lambda_1, \lambda_2$ integers, which is
impossible (Theorem \ref{multiplicity_one}).

If $a$ is odd, then we can have non-integers $\lambda_1$ and $\lambda_2$.
Step 2 in our L Space analysis (Section \ref{section_L})  applies here,
and says that in this case $\Lambda$ is
\[
\{ (an/2, -an/2) \mid n \in \Z \}.
\]
 However, the
$x$ (and $y$) projection is $\frac{a}{2}{\mathbb Z}$ and $a>2$. 
By Landau's theorem, a frame spectrum for $L^2(\Omega)$ must have a Beurling density of at least the Lebesgue measure of $\Omega$. For $ \frac{a}{2}{\mathbb Z}$, the density $\frac{2}{a}<1$. This set cannot be a frame for $L^2([t,t+1])$ (\cite{Landau}).
But if $E(\Lambda)$ is an orthonormal basis,
then it is also a frame, and so 
its projections
must also be frames
(Theorem \ref{bounded_multiplicity}).
We have an orthonormal sequence
that is not a basis.
We conclude that none of these spaces has an orthonormal basis.
\end{proof}

\subsection{The T space has no orthonormal basis}

In this subsection, we also determine the spectrality of the T-space (c.f. Figure \ref{fig_T} in Section 2).

\vspace{1em}
\begin{theorem}
  There is no orthonormal basis for the T Space.
\end{theorem}

\begin{proof}
  The Orthogonality Equation 
  \eqref{orthogonality_equation}
  for the T space has $t = -1/2$ and $t' = -1$.
  Thus
  \[
    \frac{\sin(\pi \lambda_1)} {\pi \lambda_1} = -
    e^{- \pi i \lambda_2} \; \frac{\sin(\pi \lambda_2)} {\pi \lambda_2},\ \lambda_1\lambda_2\neq 0.
  \]

  Since the left hand side is real, we have
    $e^{- \pi i \lambda_2} = \pm 1$.  But this means
    $\lambda_2$ is a nonzero integer, which means the right hand side is zero.  We must have $\sin(\pi \lambda_1) = 0$, which means that
    $\lambda_1$ is a nonzero integer.

    We have shown that this equation is true only when $\lambda_1$
    and $\lambda_2$ are integers.  But this means that $\Lambda$ itself,
    which contains at least $(0,0)$,
    contains integer pairs only.  We have shown that this is
    impossible
    (Theorem \ref{multiplicity_one}).
    Therefore the T Space has no orthonormal basis.
\end{proof}

\section{Remarks and open questions}

The paper leaves many unanswered questions concerning exponential (Riesz) bases for additive spaces. Concerning the overlapping symmetric additive space of the Lebesgue type, i.e. the component space has measure being the Lebesgue measure supported on $[t,t+1]$ where $-1/2\le t<0$,
\medskip

{\bf Question 1:} Is it true that there is no exponential orthonormal basis when $-1/2\le t<0$?

\medskip

In particular, when $t=-1/2$, we call this the {\it Plus Space}, and this is an interesting case since the  Orthogonality Equation is 
$$
 \frac {\sin(\pi \lambda_1)}
  {\pi \lambda_1} =
  -  \frac {\sin(\pi \lambda_2)}
  {\pi \lambda_2}.
$$ 
for which the phase factors all vanish.  We have the following partial result. 

\vspace{1em}
\begin{proposition}
  Suppose $\Lambda$ is a set of exponents for a Fourier frame of the Plus Space.  Then the points of $\Lambda$ are not contained in a straight line.
  \end{proposition}

\begin{proof}
It is clear that $\Lambda$ cannot lie on a vertical or a horizontal line since either one of the projections has only one point, which cannot form a frame for its component. Let $
\Lambda = \{ (\lambda, a\lambda+b): \lambda\in \Lambda_0 \}$
for some countable set $\Lambda_0\subset{\mathbb R}$. We note that  $b$ contributes only a phase factor $e_b(x)$ and we can absorb it into the function in the component space.  We can therefore assume $b=0$. Indeed, if $a>1$, we can reparametrize the $\Lambda$ as $(a^{-1}\lambda,\lambda)$ and switch the role of $x$ and $y$ axis. Furthermore, our interval is symmetric about the origin.  In this case, we can assume $0< a \le 1$. Hence, we  can assume
$$
\Lambda = \{ (\lambda, a\lambda): \lambda\in \Lambda_0 \}, \ 0<a\le 1.
$$

\vspace{1em}

Suppose that $E(\Lambda)$ forms a frame for the Plus Space. For any $f,g\in L^2[-1/2,1/2]$, we have
\begin{equation}\label{eqapp1}
\sum_{\lambda\in\Lambda_0} \left| \int_{-\frac12}^{\frac12} f(x) e_{\lambda}(x)dx+\int_{-\frac12}^{\frac12} g(x) e_{a\lambda}(x)dx\right|^2 \asymp \int_{-\frac12}^{\frac12} |f|^2dx +\int_{-\frac12}^{\frac12} |g|^2dx
\end{equation}
By a change of variable, our left hand side is equal to 
$$
\begin{aligned}
\sum_{\lambda\in\Lambda_0} \left| \int_{-\frac12}^{\frac12} f(x) e_{\lambda}(x)dx+\int_{-\frac12}^{\frac12} g(x) e_{a\lambda}(x)dx\right|^2 =& \sum_{\lambda\in\Lambda_0} \left| \int_{-\frac12}^{\frac12} f(x) e_{\lambda}(x)dx+\int_{-\frac{a}2}^{\frac{a}2} \frac1{a}g(\frac{x}{a}) e_{\lambda}(x)dx\right|^2 \\
=& \sum_{\lambda\in\Lambda_0} \left| \int_{-\frac12}^{\frac12} (f(x) + \widetilde{g}(x))e_{\lambda}(x)dx\right|^2\\
\end{aligned}
$$
  where $\widetilde{g}(x) = a^{-1}g(x/a) \chi_{[-a/2,a/2]}$ and $\widetilde{g}\in L^2[-1/2,1/2]$ since $0<a\le 1$. We also have 
\begin{equation}\label{eqapp2}
  \int_{-\frac12}^{\frac12} |\widetilde{g}(x)|^2dx = \frac{1}{a}\int_{-\frac12}^{\frac12}|g(x)|^2dx.
\end{equation}
    Note that $\Lambda$ forms a Fourier frame, which implies that $\Lambda$  (hence $\Lambda_0$) must be relatively separated\footnote{A set $\Lambda$ is relatively separated if it is a finite union of separated sets $\Lambda_i$ and $\Lambda_i$ is separated if there exists $\delta>0$ such that $|\lambda-\lambda'|\ge \delta>0$ for all $\lambda\ne \lambda'\in\Lambda_i$.} (see e.g. \cite[Proposition 2.4]{Lai}).  This implies that $E(\Lambda_0)$ must be a Bessel sequence of $L^2[-1/2,1/2]$ (i.e. it satisfies the upper bound of the frame inequality). Therefore, 
  $$
   \sum_{\lambda\in\Lambda_0} \left| \int_{-\frac12}^{\frac12} (f(x) + \widetilde{g}(x))e_{\lambda}(x)dx\right|^2 \le B \int_{-\frac12}^{\frac12} |f+\widetilde{g}|^2 dx.
  $$ 
  Combining this with (\ref{eqapp1}) and (\ref{eqapp2}), we can find a constant $C$ such that for all $f,g\in L^2[-1/2,1/2]$, 
  $$
 \int_{-\frac12}^{\frac12} |f|^2dx +a \int_{-\frac12}^{\frac12} |\widetilde{g}|^2dx \le C  \int_{-\frac12}^{\frac12} |f+\widetilde{g}|^2 dx.
  $$
  However, this is not possible since we can take $f= -\widetilde{g}\ne 0$. 
    \end{proof}
    
    It is unlikely that the Plus Space or any other overlapping symmetric additive space of Lebesgue type will admit an orthonormal basis. It would be more interesting to determine if they will admit a Riesz basis. 
    
    \medskip

{\bf Question 2:} Does there exist any exponential Riesz basis when $-1/2\le t<0$? 

\medskip 

If we look at the literature, any measure known to  admit a Fourier frame also admits a Riesz basis. It would be interesting to obtain an example of the measures as below:

    \medskip

{\bf Question 3:} Does there exist  a measure  $\mu$ that admits a Fourier frame but not a Riesz basis?

\medskip 

Our result in this paper only focuses on the addition of two measures supported on two subspaces. Our problem can also be studied for other subspaces. In \cite{ILLW}, the authors showed that finitely many additions of Lebesgue measures supported on different subspaces admit Fourier frames in the ambient space. This also shows that the area measures on boundary of polytopes admit Fourier frames. It would be interesting to determine if the area measures admit exponential Riesz bases or orthonormal bases.  In particular, 

\medskip

{\bf Question 4:} Does the length measure of the boundary of a triangle, or a square, admit an exponential Riesz basis?

\medskip

An answer to Question 4  may shed some light on the long open problem of whether the triangle admits an exponential Riesz basis.

\vspace{1em}
%

\end{document}